\documentclass[11pt]{article}
\usepackage{amsmath, amssymb, amsthm, verbatim,enumerate,bbm,color} %DIF >
\usepackage{indentfirst}

\title{Efficient Removal without Efficient Regularity}

\author{Lior Gishboliner \thanks{School of Mathematical Sciences, Tel Aviv University, Tel Aviv, 69978, Israel.}
\and Asaf Shapira \thanks{
School of Mathematical Sciences, Tel Aviv University, Tel Aviv 69978, Israel.
Email: asafico$@$tau.ac.il. Supported in part by ISF Grant 1028/16 and ERC Starting Grant 633509.}
}

\date{\today}
\allowdisplaybreaks

\theoremstyle{plain}
\newtheorem{theorem}{Theorem}[section]
\newtheorem{lemma}[theorem]{Lemma}
\newtheorem{claim}[theorem]{Claim}

\newcommand{\V}{\mathcal V}

\newcommand{\Q}{\mathcal Q}

\newcommand{\Part}{\mathcal P}
\newcommand{\F}{\mathcal F}

\newcommand{\W}{\mathcal W}

\newcommand{\poly}{\text{poly}}
\RequirePackage[normalem]{ulem} %DIF PREAMBLE
\RequirePackage{color}\definecolor{RED}{rgb}{1,0,0}\definecolor{BLUE}{rgb}{0,0,1} %DIF PREAMBLE
\begin{document}

\maketitle
\begin{abstract}
Obtaining an efficient bound for the triangle removal lemma is one of the most outstanding open problems
of extremal combinatorics. Perhaps the main bottleneck for achieving this goal is that
triangle-free graphs can be highly unstructured. For example, triangle-free graphs might have
only regular partitions (in the sense of Szemer\'edi) of tower-type size. And indeed, essentially all the graph
properties ${\cal P}$ for which removal lemmas with reasonable bounds were obtained, are such that every graph
satisfying ${\cal P}$ has a small regular partition. So in some sense, a barrier for obtaining
an efficient removal lemma for property ${\cal P}$ was having an efficient regularity lemma
for graphs satisfying ${\cal P}$.

In this paper we consider the property of being induced $C_4$-free, which also suffers from the
fact that a graph might satisfy this property but still have only regular partitions of tower-type size.
By developing a new approach for this problem we manage to overcome this barrier and thus obtain a merely exponential bound
for the induced $C_4$ removal lemma. We thus obtain the first efficient removal lemma that does not rely on an efficient version of the regularity lemma.
This is the first substantial progress on a problem raised by Alon in 2001,
and more recently by Alon, Conlon and Fox.

\end{abstract}

\section{Introduction}\label{sec:intro}

An $n$-vertex graph is $\varepsilon$-far from satisfying a property ${\cal P}$ if one should add/delete at least
$\varepsilon n^2$ edges in order to turn $G$ into a graph satisfying ${\cal P}$.
The so called {\em induced removal lemma} of Alon, Fischer, Krivelevich and Szegedy \cite{AFKS} states
that for every fixed graph $H$, if an $n$-vertex graph $G$ is $\varepsilon$-far from being
induced $H$-free, then $G$ contains at least $n^h/\mbox{Rem}_H(\varepsilon)$ induced copies of $H$, where $h=|V(H)|$ and
$\mbox{Rem}_H(\varepsilon)$ depends only on $\varepsilon$. The proof of this lemma in \cite{AFKS} supplied extremely weak
bounds for $\mbox{Rem}_H(\varepsilon)$, which were later improved
by Conlon and Fox \cite{ConlonFo12}. However, even these improved bounds are of tower-type\footnote{We use $\mbox{tower(x)}$ for a tower of exponents of height $x$, so $\mbox{tower(3)}=2^{2^2}$.
The original proof of the induced removal lemma in \cite{AFKS} gave only wowzer-type bounds, where wowzer is the iterated-tower function.}.

Alon \cite{subgraphs} asked for which graphs $H$ we have $\mbox{Rem}_H(\varepsilon)=\poly(1/\varepsilon)$, that is, for which graphs $H$
can we obtain polynomial bounds for the induced removal lemma. This question was
addressed by Alon and the second author \cite{induced} who resolved this problem for all
graphs $H$ save for $P_4$ (the path on $4$ vertices) and $C_4$ (the $4$-cycle).
The former case was recently solved by Alon and Fox \cite{AF}, who proved that $\mbox{Rem}_{P_4}(\varepsilon)=\poly(1/\varepsilon)$.
They further asked to determine if $\mbox{Rem}_{C_4}(\varepsilon)=\poly(1/\varepsilon)$. This problem was also later raised by Conlon
and Fox \cite{ConlonFo13}.

Prior to this work the best bound for $\mbox{Rem}_{C_4}(\varepsilon)$ was the same tower-type bound that holds
for all graphs $H$. As we explain in the next subsection, the reason is that this problem seemed to lie just
outside the realm of the known techniques for proving efficient bounds for graph removal lemmas.
Our main result in this paper makes the first substantial progress on this problem, by improving the tower-type bound into
an exponential one.

\begin{theorem}\label{thm:c4}
If an $n$-vertex graph $G$ is $\varepsilon$-far from being induced $C_4$-free, then $G$ contains at least $n^4/2^{(1/\varepsilon)^c}$ induced copies of $C_4$, where $c$ is an absolute constant.
\end{theorem}

We conjecture that the exponential bound in Theorem \ref{thm:c4} can be further improved to a polynomial one.

Given a (possibly infinite) family of graphs ${\cal F}$, we say that a graph is induced ${\cal F}$-free if it is induced
$H$-free for every $H \in {\cal F}$.
%Observe that for infinite families ${\cal F}$ it is not a priori clear that if $G$ is $\varepsilon$-far from being induced ${\cal F}$-free
%that $G$ should contain {\em any} constant size (that might depend on $\varepsilon$) subgraph that is not induced ${\cal F}$-free.
Observe that for infinite families ${\cal F}$ it is not a priori clear that a graph which is $\varepsilon$-far from being induced ${\cal F}$-free
should contain {\em any} constant size (that might depend on $\varepsilon$) subgraph that is not induced ${\cal F}$-free.
Such a result was obtained by Alon and the second author~\cite{hereditary}, who extended the result of \cite{AFKS} by showing
that for every family of graphs ${\cal F}$, there is a function $\mbox{Rem}_{\cal F}(\varepsilon)$, so that if $G$ is $\varepsilon$-far from
being induced ${\cal F}$-free, then a random subset of $\mbox{Rem}_{\cal F}(\varepsilon)$ vertices from $V(G)$ is not induced ${\cal F}$-free
with probability at least (say\footnote{It is easy to see that if ${\cal F}=\{H\}$, this way of defining $\mbox{Rem}_{\cal F}(\varepsilon)$ is equivalent (up to polynomial factors) to the induced removal lemma of \cite{AFKS}, as we stated it above.}) $2/3$. Needless to say that as in \cite{AFKS}, the bounds for $\mbox{Rem}_{\cal F}(\varepsilon)$ given by \cite{hereditary} are also (at least) of tower-type.

It is natural to ask if Theorem \ref{thm:c4} can be extended to properties defined by forbidding a family of graphs ${\cal F}$, one of which is $C_4$.
The most notable and natural example is the property of being {\em chordal}, which is the property of not containing an induced cycle of length at least
$4$. Previously, the best bound for this problem was the tower-type bound which follows from the general result of~\cite{hereditary}. Here we obtain the following improved bound.

\begin{theorem}\label{thm:chordal}
If an $n$-vertex graph $G$ is $\varepsilon$-far from being chordal, then for some $4 \leq \ell \leq O(\varepsilon^{-18})$,
$G$ contains at least $n^{\ell}/2^{(1/\varepsilon)^c}$ induced copies of $C_{\ell}$,
where $c$ is an absolute constant.
\end{theorem}

While Theorem \ref{thm:chordal} asserts that if $G$ is $\varepsilon$-far from being chordal it must contain an induced
cycle of length $\mbox{poly}(1/\varepsilon)$, it only implies that a sample of vertices of size $2^{(1/\varepsilon)^c}$
contains an induced cycle with probability at least $2/3$. We do believe, however, that this exponential bound can be
further improved to a polynomial one.

It is now natural to ask if Theorem \ref{thm:chordal} can be further extended to an arbitrary family of graphs ${\cal F}$, one of which is $C_4$.
As our final theorem shows, this is not the case in a very strong sense.

\begin{theorem}\label{thm:hard0}
For every (decreasing) function $g\colon(0,1/2) \rightarrow \mathbb{N}$ there is a family of graphs ${\cal F}={\cal F}(g)$
so that $C_4 \in {\cal F}$ and yet $\mathrm{Rem}_{\cal F}(\varepsilon) \geq g(\varepsilon)$.

In fact, for every (small enough) $\varepsilon >0$ and every $n \geq n_0(\varepsilon)$, there is an $n$-vertex graph $G$ which is $\varepsilon$-far from being
induced ${\cal F}$-free, and yet does not contain an induced copy of {\bf any} $F \in {\cal F}$ on fewer than $g(\varepsilon)$ vertices.
\end{theorem}

\subsection{Relation to prior works}\label{subsec:long}

In this subsection we would like to explain why in Theorem \ref{thm:c4} we managed
to overcome for the first time a natural barrier, which was the main reason why one
could not derive Theorem \ref{thm:c4} via techniques that were previously used for
proving graph removal lemmas. For simplicity we will start by discussing the triangle removal lemma,
that is, the special case of the induced removal lemma when $H=K_3$. The original proof
of the triangle removal lemma \cite{RuzsaSz76} relied on the famous regularity lemma of Szemer\'edi \cite{Szemeredi78},
which is one of the most powerful tools for tackling problems
in extremal graph theory. It states that for every $\varepsilon >0$ there is an $M=M(\varepsilon)$ so that every
graph has an $\varepsilon$-regular partition of order at most $M$ (see \cite{RodlSh} for the precise definitions related to graph regularity).
Since Szemer\'edi's proof only established that $M(\varepsilon) \leq \mbox{tower}(1/\varepsilon^5)$, this approach for proving the triangle
removal lemma only gave the very weak bound $\mbox{Rem}_{K_3}(\varepsilon) \leq \mbox{tower}(\mbox{poly}(1/\varepsilon))$.
Gowers' celebrated result \cite{Gowers}, which states that $M(\varepsilon) \geq \mbox{tower}(\mbox{poly}(1/\varepsilon))$, implies
that one cannot get a better bound for $\mbox{Rem}_{K_3}(\varepsilon)$ via the regularity lemma.
In a major breakthrough, Fox \cite{Fox11} managed to prove the triangle removal lemma
while avoiding Szemer\'edi's version of the regularity lemma, thus showing that $\mbox{Rem}_{K_3}(\varepsilon) \leq \mbox{tower}(O(\log 1/\varepsilon))$.
A different formulation of his proof was later given in \cite{ConlonFo13} and \cite{MoshkovitzSh16}. The latter proof
shows that Fox's result can be derived from a variant of the regularity lemma. Unfortunately,
it was shown in \cite{MoshkovitzSh16} that this variant of the regularity lemma must also produce
partitions of tower-type size. Hence this approach does not seem to allow one to prove (say) exponential bounds for
the triangle removal lemma.

Although the best known bounds for the triangle removal are of tower-type, there are families of graphs
${\cal F}$ for which one {\em can} prove much better (non-tower-type) bounds for $\mbox{Rem}_{\cal F}(\varepsilon)$, that is, for
the removal lemma of induced ${\cal F}$-freeness. One example is the result of Alon and Fox \cite{AF} mentioned above regarding induced $P_4$-freeness.
The main point we would like to make is that
all these improved bounds (save for one case discussed below) were not obtained by avoiding the regularity lemma.
Instead, they still (implicitly or explicitly) used the regularity lemma, but relied on the fact that induced ${\cal F}$-free
graphs have much smaller $\varepsilon$-regular partitions. For example, the result of Alon and Fox \cite{AF} regarding induced $P_4$-freeness
can be derived from the fact that every induced $P_4$-free graph has an $\varepsilon$-regular partition of size $\mbox{poly}(1/\varepsilon)$.
See \cite{GS} for a proof of this and other related results.

It is now natural to ask if one can use the above approach in order to obtain better bounds for the triangle removal lemma.
Unfortunately, there are bipartite versions of Gowers' \cite{Gowers} lower bound for the regularity lemma, as well as for
the variant of the regularity lemma introduced in \cite{MoshkovitzSh16}. Therefore, a graph can be triangle-free but still only have
regular partitions of tower-type size. This means that any proof of the triangle removal lemma that relies on (one of the above versions of) the regularity lemma is bound to produce tower-type bounds.

With regard to induced $C_4$-freeness, it is easy to see that every split graph is induced $C_4$-free, where a split
graph is a graph whose vertex set can be partitioned into two sets, one inducing a complete graph and the other an independent set. 
This means that if we take a bipartite version of Gowers' lower bound \cite{Gowers} (or of the one from \cite{MoshkovitzSh16}),
and put a complete graph on one of the vertex sets, we get an induced $C_4$-free graph that has only regular partitions of tower-type size.
In particular, arguments similar to those that were previously used in order to devise efficient removal lemmas cannot
give better-than-tower-type bounds for this problem.

Summarizing the above discussion, Theorem \ref{thm:c4} is the first example
showing that one {\em can} obtain an efficient removal lemma for a property ${\cal P}$,
even though graphs satisfying ${\cal P}$ might have only regular partitions of tower-type size.
To do this, our proof is the first removal lemma that avoids using the regularity lemma or one of its variants (save for the example discussed below).
We are hopeful that bounds similar to those obtained in Theorem \ref{thm:c4} can
be obtained for removal lemmas of other properties for which the best known bounds are of tower-type, most notably for triangle freeness.

%\begin{remark}

Let us end the discussion by describing the only previous example of a removal lemma that was obtained while avoiding
a regularity lemma, and how it differs from Theorem \ref{thm:c4}.
In 1984 Erd\H{o}s \cite{Erdos84} (implicitly) conjectured that $k$-colorability has a removal lemma, that is,
that if $G$ is $\varepsilon$-far from being $k$-colorable then a sample of $C_k(\varepsilon)$ vertices from $V(G)$ spans a non-$k$-colorable subgraph with probability at least $2/3$. This was first verified by R\"odl and Duke \cite{RD} who used the regularity lemma in order to obtain 
a tower-type bound for $C_k(\varepsilon)$.
This tower-type bound was dramatically improved by Goldreich, Goldwasser and Ron \cite{GGR},
who obtained a new proof of this result (as well as for similar {\em partition problems}) that avoided the regularity lemma
and thus gave a polynomial bound for $C_k(\varepsilon)$. Let us try\footnote{See also Subsection 8.3.2 of Goldreich's upcoming book \cite{Goldreich} for a similar attempt.} to explain why $k$-colorability differs from triangle-freeness or
induced $C_4$-freeness. First, as opposed to
these two properties which are {\em local}, the partition properties of \cite{GGR} are {\em global}.
Perhaps the best way to see this is from the perspective of graph homomorphisms: triangle-freeness means that there is no edge-preserving mapping
from the vertices of the triangle to the vertices of $G$, while $3$-colorability means that there is such a mapping from the vertices of
$G$ to the vertices of the triangle\footnote{In the language of graph limits, this is the distinction between left and right homomorphisms, see \cite{Lovasz}.}. The second difference, which is more important for our quantitative investigation here, is that 
$k$-colorability is defined using global edge counts (i.e. having no edges inside a vertex partition into $k$ sets). 
This can explain (at least in hindsight), why one does not need any structure theorem in order to handle this property. 
Instead one can rely on sampling arguments that boil down to estimating various edge densities (this is not to say that devising such proofs is an easy task!). It appears that arguments of this sort cannot be used to prove removal lemmas for local properties
such as triangle freeness or induced $C_4$-freeness.

%\end{remark}

\subsection{Paper overview}
The main idea of the proof is to show that (very roughly speaking) every induced $C_4$-free graph is a split graph.
To be more precise, every\footnote{It is known \cite{Steger} that {\em most} induced $C_4$-free graphs are split graphs. We stress that in our setting
we have to deal with {\em every} induced $C_4$-free graph, not just typical ones!} induced $C_4$-free graph is close to being a union of an independent set and few cliques, so that
the bipartite graphs between these cliques are highly structured. Note that we have no guarantee on the structure
of the bipartite graph connecting the independent set and the cliques\footnote{This unstructured part is unavoidable
due to the example we mentioned earlier of putting Gowers' construction between a clique and an independent set.}.
Towards this goal, in Section \ref{sec:prelim} we describe some preliminary lemmas, mostly regarding the structure
of bipartite graphs that do not contain an induced matching of size $2$.
In Section \ref{sec:structure} we give the main partial structure theorem, stated as Lemma \ref{lemma:main}.
In the course of the proof we will make a surprising application of the main result of
Goldreich, Goldwasser and Ron \cite{GGR}. In Section \ref{sec:main} we give the proofs of Theorems \ref{thm:c4} and \ref{thm:chordal}.
We will make use of the structure theorem from Section \ref{sec:structure} but will also have to deal
with the (unavoidable) {\em unstructured} part of the graph. This will be done in Lemma \ref{lemma:ind_set}.
Finally, in Section \ref{sec:example}, we give the proof of Theorem \ref{thm:hard0}.
We will make no effort to optimize the constant $c$ appearing in Theorems \ref{thm:c4} and \ref{thm:chordal}.

\section{Forbidding an induced $2$-matching}\label{sec:prelim}

Our goal in this section is to introduce several definitions and prove Lemma \ref{lemma:homog_part_strong} stated below,
regarding graphs not containing induced matchings of size $2$ of a specific type, which we now formally define.
Let $G$ be a graph and let $X,Y \subseteq V(G)$ be disjoint sets of vertices.
%We say that $(X,Y)$ is {\em homogeneous} if the bipartite graph between $X$ and $Y$ is either complete or empty.
An {\em induced copy of $M_2$} in $(X,Y)$ is an (unordered) quadruple $x,x',y,y'$ such that $x,x' \in X$, $y,y' \in Y$, $(x,y),(x',y') \in E(G)$ and $(x,y'),(x',y) \notin E(G)$. We say that $(X,Y)$ is {\em induced $M_2$-free} if it does not contain induced copies of $M_2$ as above.
%We say that $(X,Y)$ is {\em $\varepsilon$-far} from being induced $M_2$-free if one has to add/delete at least $\varepsilon|X||Y|$ of the edges between $X$ and $Y$ to make $(X,Y)$ induced $M_2$-free. (????where is the best place to say this???? Moved this to before Lemma \ref{lemma:M2_removal})
Observe that if $X$ and $Y$ are cliques then $G[X \cup Y]$ is induced $C_4$-free if and only if $(X,Y)$ is induced $M_2$-free. For $x \in X$, we denote $N_Y(x) = \{y \in Y : (x,y) \in E(G)\}$.

\begin{claim}\label{prop:nested}
$(X,Y)$ is induced $M_2$-free if and only if there is an enumeration $x_1,\dots,x_m$ of the elements of $X$ such that
$N_Y(x_i) \subseteq N_Y(x_j)$ for every $1 \leq i < j \leq m$.
\end{claim}

\begin{proof}
Observe that $(X,Y)$ contains an induced $M_2$ if and only if there are $x,x' \in X$ for which there exist $y \in N_Y(x) \setminus N_Y(x')$ and $y' \in N_Y(x') \setminus N_Y(x)$. Therefore,
$(X,Y)$ is induced $M_2$-free if and only if for every $x,x' \in X$ it holds that either $N_Y(x) \subseteq N_Y(x')$ or $N_Y(x') \subseteq N_Y(x)$. Consider the poset on $X$ in which $x$ precedes $x'$ if and only if $N_Y(x) \subseteq N_Y(x')$. This poset is a linear ordering. Enumerate the elements of $X$ from minimal to maximal to get the required enumeration.
\end{proof}

We say that $(X,Y)$ is {\em homogeneous} if the bipartite graph between $X$ and $Y$ is either complete or empty. We say that a partition ${\cal P} = \{P_1,\dots,P_r\}$ of a set $V$ is an {\em equipartition} if $\left| |P_i| - |P_j| \right| \leq 1$ for every $1 \leq i,j \leq r$.

\begin{lemma}\label{lemma:homog_pair_part}
If $(X,Y)$ is induced $M_2$-free then for every integer $r \geq 1$ there is an equipartition
$X = X_1 \cup \dots \cup X_r$ and a partition
$Y = Y_1 \cup \dots \cup Y_{r + 1}$ such that $(X_i,Y_j)$ is homogeneous for every $1 \leq i \leq r$ and $1 \leq j \leq r + 1$ satisfying $i \neq j$.
\end{lemma}

\begin{proof}
Let $x_1,\dots,x_m$ be the enumeration of the elements of $X$ from Claim \ref{prop:nested}. For $1 \leq i \leq r$ define
$X_i = \{ x_j : \frac{(i-1)m}{r} < j \leq \frac{im}{r} \}$. Here we assume, for simplicity of presentation, that $|X|$ is divisible by $r$; if that is not the case then we partition $X$ into ``consecutive intervals'' of sizes
$\big\lfloor \frac{|X|}{r} \big\rfloor$ and
$\big\lceil \frac{|X|}{r} \big\rceil$.
Let now $y_1,...,y_n$ be an enumeration of the elements of $Y$ with the property that for every $x \in X$, the set $N_Y(x)$ is a ``prefix'' of the enumeration, that is, so that $N_Y(x) = \{y_1,\dots,y_k\}$ for some $0 \leq k \leq n$.
Define $Y_1 = N_Y(x_{m/r})$, $Y_i = N_Y(x_{im/r}) \setminus N_Y(x_{(i-1)m/r})$ for $i = 2,\dots,r$ and
$Y_{r+1} = Y \setminus N_Y(x_m)$.

It remains to show that $(X_i,Y_j)$ is homogeneous for every $i \neq j$. Assume first that $i < j$. Then for every $x \in X_i$ we have
$N_Y(x) \subseteq N_Y(x_{im/r}) \subseteq N_Y(x_{(j-1)m/r})$.
By the definition of $Y_j$ we have $Y_j \cap \nolinebreak N_Y(x_{(j-1)m/r}) = \emptyset$. Thus,
$Y_j \cap N_Y(x) = \emptyset$ for every $x \in X_i$, implying that the bipartite graph $(X_i,Y_j)$ is empty. Now assume that $i > j$. For every $x \in X_i$ we have
$N_Y(x_{jm/r}) \subseteq  \nolinebreak N_Y(x_{(i-1)m/r}) \subseteq \nolinebreak N_Y(x)$. By the definition of $Y_j$ we have
$Y_j \subseteq N_Y(x_{jm/r})$. Thus, $Y_j \subseteq N_Y(x)$ for every
$x \in X_i$, implying that the bipartite graph $(X_i,Y_j)$ is complete.
\end{proof}

For two partitions ${\cal P}_1, {\cal P}_2$ of the same set, we say that ${\cal P}_2$ is a
{\em refinement} of ${\cal P}_1$ if every part of ${\cal P}_2$ is contained in one of the parts of
${\cal P}_1$.
A vertex partition
${\cal P}$ of an $n$-vertex graph $G$ is called {\em $\delta$-homogeneous} if the sum of $|U||V|$ over all non-homogeneous unordered distinct pairs
$U,V \in \mathcal{P}$ is at most $\delta n^2$. It is easy to see that a refinement of a $\delta$-homogeneous partition is itself $\delta$-homogeneous.

\begin{lemma}\label{lemma:reg_part}
Let $\delta > 0$, let $G$ be an $n$-vertex graph and let
$V(G) = X_1 \cup \dots \cup X_k$
be a partition such that $X_1,\dots,X_k$ are cliques and $(X_i,X_j)$ is induced $M_2$-free for every $1 \leq i < j \leq k$. Then there is a $\delta$-homogeneous partition which refines $\{X_1,\dots,X_k\}$ and has at most
$k \left( 2/\delta \right)^{k}$ parts.
\end{lemma}

\begin{proof}
For every $1 \leq i < j \leq k$, we apply Lemma \ref{lemma:homog_pair_part} to $(X_i,X_j)$ with parameter $r = \frac{1}{\delta}$ to get partitions $\Part_{i,j}$ of $X_i$ and $\Part_{j,i}$ of $X_j$,
$\Part_{i,j} = \{X_{i,j}^{1},...,X_{i,j}^{r}\}$,
$\Part_{j,i} = \{ X_{j,i}^{1},...,X_{j,i}^{r+1} \}$, such that $\Part_{i,j}$ is an equipartition and $(X_{i,j}^p, X_{j,i}^q)$
is homogeneous for every $p \neq q$.
Note that
\begin{equation}\label{eq1}
\sum_{p=1}^{r}{|X_{i,j}^p||X_{j,i}^p|} =
\sum_{p=1}^{r}{\frac{1}{r}|X_i||X_{j,i}^p|} \leq
\frac{1}{r}|X_i||X_j| = \delta |X_i||X_j|.
\end{equation}
For every $i = 1,...,k$, define $\Part_i$ to be the common refinement of the partitions $(\Part_{i,j})_{1 \leq j \leq k, \; j \neq i}$.
We have
$|\Part_i| \leq (r+1)^{k-1} \leq
\left( 2/\delta \right)^{k}$. The partition
$\Part := \bigcup_{i=1}^{k}{\Part_i}$ refines $\{X_1,\dots,X_k\}$ and has at most
$k \left( 2/\delta \right)^{k}$ parts. For every
$U,V \in \mathcal{P}$, if $(U,V)$ is not homogeneous, then there are
$1 \leq i < j \leq k$ and $1 \leq p \leq r$ such that $U \subseteq X_{i,j}^p$ and
$V \subseteq X_{j,i}^p$. This follows from the fact that $X_1,\dots,X_k$ are cliques and the property of the partitions $(\mathcal{P}_{i,j})_{1 \leq i \neq j \leq k}$. By (\ref{eq1}), we have
$$
\sum_{1 \leq i < j \leq k}{\sum_{p=1}^{r}{|X_{i,j}^p||X_{j,i}^p|}} \leq
\delta\sum_{1 \leq i < j \leq k}{|X_i||X_j|} \leq \delta n^2,
$$
implying that $\mathcal{P}$ is $\delta$-homogeneous, as required.
\end{proof}

%We say that $(X,Y)$ is
%{\em $\varepsilon$-far} from being induced $M_2$-free if one has to add/delete at least $\varepsilon|X||Y|$ of the edges between $X$ and $Y$ to make $(X,Y)$ induced $M_2$-free.
%\begin{lemma}\label{lemma:2K2_removal}
%Let $G = (X \cup Y, E)$ be a bipartite graph, $|X| = |Y| = n$. Suppose that one must change at least $\varepsilon n^2$ of the edges/non-edges between $X$ and $Y$ to remove all copies of $2K_2$. Then $G$ contains at least
%$\varepsilon^d n^4$ copies of $2K_2$.
%\end{lemma}

%Observe that if $X$ and $Y$ are cliques then $G[X \cup Y]$ is induced $C_4$-free if and only if $(X,Y)$ is induced $M_2$-free.

\begin{lemma}\label{lemma:homog_part_strong}
Let $\delta > 0$, let $G$ be an $n$-vertex graph and let
$V(G) = X_1 \cup \dots \cup X_k$
be a partition such that $X_1,\dots,X_k$ are cliques and $(X_i,X_j)$ is induced $M_2$-free for every $1 \leq i < j \leq k$.
Then there is a set $Z \subseteq V(G)$ of size $|Z| < \delta n$, a partition
$V(G) \setminus Z = Q_1 \cup \dots \cup Q_q$ which refines
$\{X_1 \setminus Z, \dots, X_k \setminus Z\}$ and subsets
$W_i \subseteq Q_i$ such that the following hold.
%Then there is a $\delta$-homogeneous partition $\mathcal{P}$ refining $\{X_1,\dots,X_k\}$, a set of parts $\mathcal{Q} = \{Q_1,\dots,Q_q\} \subseteq \mathcal{P}$ and subsets
%$W_i \subseteq Q_i$ such that the following hold.
\begin{enumerate}
\item The sum of $|Q_i||Q_j|$ over all non-homogeneous pairs $(Q_i,Q_j)$, $1 \leq i < j \leq q$, is at most $\delta n^2$.
%\item The sum $\sum{|Q_i||Q_j|}$, taken over all non-homogeneous pairs $(Q_i,Q_j)$, is at most $\delta n^2$.
\item $|W_i| \geq (\delta/2k)^{10k^2}n$ for every $1 \leq i \leq q$ and $(W_i,W_j)$ is homogeneous for every $1 \leq i < j \leq q$.

\end{enumerate}
\end{lemma}
\begin{proof}
Apply Lemma \ref{lemma:reg_part} to $G$ with parameter $\delta$ to obtain a $\delta$-homogeneous partition $\mathcal{P}$ which refines $\{X_1,\dots,X_k\}$.
Define
$\mathcal{Q} =
\{ U \in \mathcal{P} : |U| \geq \frac{\delta n}{|\Part|} \}$ and write $\mathcal{Q} = \{Q_1,\dots,Q_q\}$. Then Item 1 holds since $\mathcal{P}$ is $\delta$-homogeneous.
Setting $Z = \bigcup_{U \in \mathcal{P} \setminus \mathcal{Q}}{U}$, notice that $\Q$ refines $\{X_1 \setminus Z, \dots, X_k \setminus Z\}$ and that
$|Z| < |\mathcal{P}| \cdot \frac{\delta n}{|\mathcal{P}|} = \delta n$.
%Lemma \ref{lemma:reg_part} gives
%\begin{equation}\label{eq:partition_size_1}
%|\Part| \leq k (4/\delta)^k.
%% \geq \left( \frac{\delta}{2} \right)^k\frac{\delta n}{k}.
%\end{equation}
Apply Lemma \ref{lemma:reg_part} to $G$ again (with respect to the same partition $\{X_1,\ldots,X_k\}$), now with parameter
$\delta' := \frac{\delta^2}{8|\mathcal{P}|^4}$, to get a $\delta'$-homogeneous partition $\mathcal{V}$ with at most
%$k(16|\Part|^4/\delta^2)^k \leq (2k|\Part|/\delta)^{4k}$ parts.
$k(16|\Part|^4/\delta^2)^k$ parts.
%\leq \nolinebreak k^5(16/\delta)^k$.
Let $\W$ be the common refinement of $\Part$ and $\V$ and note that $\W$ is $\delta'$-homogeneous since it is a refinement of $\mathcal{V}$. Moreover,
\begin{equation}\label{eq:partition_size_2}
|\W| \leq |\Part| \cdot |\V| \leq |\Part| \cdot k(16|\Part|^4/\delta^2)^k.
%|\W| \leq |\Part| \cdot |\V| \leq |\Part| \cdot (2k|\Part|/\delta)^{4k}.
%\leq (2k|\Part|)^{5k} \leq (4k/\delta)^{5k^2}.
\end{equation}

For each $1 \leq i \leq q$, define $\W_i = \{W \in \W : W \subseteq Q_i\}$, choose a vertex $w_i \in Q_i$ uniformly at random and let $W_i \in \W_i$ be such that $w_i \in W_i$. We will show that with positive probability, the sets $W_1,...,W_q$ satisfy the statement in Item 2.
For $1 \leq i \leq q$, the probability that
$|W_i| < \frac{|Q_i|}{2q\left| \W \right|}$ is smaller than
$\frac{|\W| \cdot \frac{|Q_i|}{2q\left| \W \right|}}{|Q_i|} = \frac{1}{2q}$. By the union bound, with probability larger than $\frac{1}{2}$, every
$1 \leq i \leq q$ satisfies
$$|W_i| \geq \frac{|Q_i|}{2q|\W|} \geq
\frac{\big( \frac{\delta^2}{16|\mathcal{P}|^4} \big)^{k} \delta n }
{2k|\mathcal{P}|^{3}} \geq
\frac{\delta^{3k}n}{k(2|\mathcal{P}|)^{7k}} \geq
\frac{\delta^{3k}n}{k2^k(2/\delta)^{7k^2}} \geq
\left( \frac{\delta}{2k} \right)^{10k^2}n\;,
$$
where in the second inequality we used $|Q_i| \geq \frac{\delta n}{|\mathcal{P}|}$,
$q \leq |\mathcal{P}|$ and \eqref{eq:partition_size_2}, and in the fourth inequality we used the bound on $|\mathcal{P}|$ given by Lemma \ref{lemma:reg_part}.
For $1 \leq i < j \leq q$, the probability that the pair $(W_i,W_j)$ is not homogeneous is
$$
\sum{\frac{|W||W'|}{|Q_i||Q_j|}} \leq \frac{4|{\cal P}|^2}{\delta^2n^2}\sum{|W||W'|}\leq \frac{4|{\cal P}|^2}{\delta^2n^2}\cdot \delta' n^2 \leq \frac{1}{2|{\cal P}|^2}\;,
$$
where the sums are taken over all non-homogeneous pairs $(W,W') \in \W_i \times \W_j$, the first inequality uses
$|\Q_i|,|\Q_j| \geq \frac{\delta n}{2|\Part|}$ and the second the fact that $\W$ is $\delta'$-homogeneous.
By the union bound, with probability at least
$1 - \binom{q}{2}\frac{1}{|\Part|} \geq 1 - \binom{|\Part|}{2}\frac{1}{|\Part|} > \frac{1}{2}$, all pairs $(W_i,W_j)$ are homogeneous. We conclude that Item 2 holds with positive probability.
\end{proof}
%\begin{lemma}\label{lemma:2K2_removal}
%Let $G = (X \cup Y, E)$ be a bipartite graph. Suppose that one must change at least $\varepsilon|X||Y|$ of the edges/non-edges between $X$ and $Y$ to remove all copies of $2K_2$. Then $G$ contains at least
%$\varepsilon^C |X|^2 |Y|^2$ copies of $2K_2$.
%\end{lemma}
%\begin{proof}
%Put $n = |X|$, $m = |Y|$. The proof is by induction on $n$. We need the following inequality.
%\begin{claim}
%For every $a,b,x,y \in [0,1]$ and $p \geq 4$ we have
%\begin{equation*}
%\left( abx + (1-a)(1-b)y \right)^p \leq \left( a^2b^2x^p + (1-a)^2(1-b)^2y^p \right)\left( ab + (1-a)(1-b) \right)^{p-3}
%\end{equation*}
%\end{claim}
%\end{proof}

\section{A partial structure theorem for $C_4$-free graphs}\label{sec:structure}

Our main goal in this section is to prove Lemma \ref{lemma:main} stated below, which
gives an approximate partial structure theorem for induced $C_4$-free graphs.
The ``approximation'' will be due to the fact that the graph will only be close to having a certain
nice structure, while the ``partial'' will be since there will be a (possibly) big part of the graph
about which we will have no control. As we discussed in Section \ref{sec:intro}, this partialness is
unavoidable as evidenced by split graphs.

In addition to the lemmas from the previous section, we will also need the following theorems of
Goldreich, Goldwasser and Ron \cite{GGR} and of Gy\'{a}rf\'{a}s, Hubenko and Solymosi \cite{GHS}.
In both cases, $\omega(G)$ denotes maximum size of a clique in $G$.

\begin{theorem}[\cite{GGR}]\label{thm:GGR}
For every $\varepsilon \in (0,1)$ there is $q_{\ref{thm:GGR}}(\varepsilon) = O(\varepsilon^{-5})$ with the following property. Let
$\rho \in (0,1)$ be such that
$\varepsilon < \rho^2/2$ and let $G$ be a graph
%on $n \geq q_{\ref{thm:GGR}}(\varepsilon)$ vertices
which is $\varepsilon$-far from satisfying $\omega(G) \geq \rho n$.
Suppose $q \geq q_{\ref{thm:GGR}}(\varepsilon)$ and let $Q \in \binom{V(G)}{q}$ be a randomly chosen set of $q$ vertices of $G$.
Then with probability at least $\frac{3}{4}$ we have
$\omega(G[Q]) < (\rho - \frac{\varepsilon}{2})q$.
%\begin{enumerate}
%\item If $\omega(G) \geq \rho n$ then
%$\omega(G[Q]) \geq (\rho - \frac{\varepsilon}{2})q$ with probability at least $\frac{3}{4}$.
%\item if $G$ is $\varepsilon$-far from satisfying
%$\omega(G) \geq \rho n$ then $\omega(G[Q]) < (\rho - \frac{\varepsilon}{2})q$ with probability at least $\frac{3}{4}$.
%\end{enumerate}
\end{theorem}
%(We now only use the negative direction - Item 2. Delete Item 1???)

\begin{theorem}[\cite{GHS}]\label{thm:GHS}
Every induced $C_4$-free graph with $n$ vertices and at least $\alpha n^2$ edges satisfies $\omega(G) \geq 0.4\alpha^2 n$.
\end{theorem}

Let use derive the following important corollary of the the above two theorems.
For a non-empty set $X \subseteq V(G)$, define $d(X) = e(X) / \binom{|X|}{2}$, where $e(X)$ is the number of edges of $G$ with both endpoints in $X$.

\begin{lemma}\label{lemma:clique}
Let $\alpha \in [0,\frac{1}{2})$ and let $G$ be a graph on $n$ vertices with at least $\alpha n^2$ edges. Then for every $\beta \in (0,1)$, either $G$ contains $\Omega(\alpha^{80}\beta^{20} n^4)$ induced copies of $C_4$ or there is a set
$X \subseteq V(G)$ with $|X| \geq 0.1\alpha^2 n$ and $d(X) \geq 1 - \beta$.
\end{lemma}

\noindent
In the proof of Lemma \ref{lemma:clique} we need the following simple fact.

\begin{claim}\label{claim:density_test}
Let $\alpha \in (0,1)$ and let $G$ be a graph with $n$ vertices and at least $\alpha n^2$ edges. Then for every $r \geq \frac{100}{\alpha^2}$, a sample of $r$ vertices from $G$ spans at least $\frac{\alpha}{2}r^2$ edges with probability at least $\frac{2}{3}$.
\end{claim}
The proof of Claim \ref{claim:density_test} is a standard application of the second moment method (see e.g. \cite{Prob_Method}), and is thus omitted.
\begin{proof}[Proof of Lemma \ref{lemma:clique}]
Set $\rho = 0.1\alpha^2$,
$\varepsilon = \frac{\rho^2\beta}{4} = \frac{\alpha^4\beta}{400}$ and
$r =
\max\{ q_{\ref{thm:GGR}}(\varepsilon), \frac{100}{\alpha^2}\}$.
By Theorem \ref{thm:GGR} we have $r = O(\alpha^{-20}\beta^{-5})$.
We assume that there is no $X \subseteq V(G)$ with $|X| \geq 0.1\alpha^2 n$ and
$d(X) \geq 1 - \beta$, and prove that $G$ contains $\Omega(\alpha^{80}\beta^{20} n^4) $ induced copies of $C_4$.
Let $X \subseteq V(G)$ be such that $|X| \geq \rho n$. Since $d(X) \leq 1 - \beta$, we have
$\binom{|X|}{2} - e(G) \geq \beta \binom{|X|}{2} \geq \beta \frac{|X|^2}{4} \geq \frac{\rho^2\beta}{4}n^2 = \varepsilon n^2$. This shows that $G$ is $\varepsilon$-far from containing a clique of size $\rho n$ or larger. By our choice of $r$ via Theorem \ref{thm:GGR}, a random sample $R$ of $r$ vertices of $G$ satisfies $\omega(G[R]) < (\rho - \frac{\varepsilon}{2})r < 0.1\alpha^2r$ with probability at least $\frac{2}{3}$. By Claim \ref{claim:density_test}, we also have
$e(R) > \frac{\alpha}{2}r^2$ with probability at least $\frac{2}{3}$. So with probability at least $\frac{1}{3}$ we have both $\omega(G[R]) < 0.1\alpha^2 r$ and $e(R) > \frac{\alpha}{2}r^2$. If both events happen, then $G[R]$ must contain an induced copy of $C_4$, by Theorem \ref{thm:GHS}. We conclude that $G$ contains at least
$\frac{1}{3}\binom{n}{r}/\binom{n-4}{r-4} = \frac{1}{3}\binom{n}{4}/\binom{r}{4} = \Omega( \alpha^{80}\beta^{20}n^4 )$ induced copies of $C_4$.
\end{proof}

The last ingredient we need is the following result of Alon, Fischer and Newman \cite{AFN}.
For a pair of disjoint vertex sets $X,Y$, we say that $(X,Y)$ is {\em $\varepsilon$-far} from being induced $M_2$-free if one has to add/delete at least $\varepsilon|X||Y|$ of the edges between $X$ and $Y$ to make $(X,Y)$ induced $M_2$-free.

\begin{lemma}[\cite{AFN}]\label{lemma:M2_removal}
There is an absolute constant $d > 0$ such that the following holds. If
$(X,Y)$ is $\varepsilon$-far from being induced $M_2$-free then $(X,Y)$ contains at least $\varepsilon^d |X|^2|Y|^2$ induced copies of $M_2$.
\end{lemma}

The following is the key lemma of this section.
Note that it gives us a lot of information about $G[Y]$ and $G[X_1 \cup \cdots \cup X_k]$ but no information
about the bipartite graph connecting $X_1 \cup \cdots \cup X_k$ and $Y$.

\begin{lemma}\label{lemma:main}
There is an absolute constant $c > 0$, such that for every
$\alpha,\gamma \in (0,1)$, every $n$-vertex graph $G$ either contains
$\Omega(\alpha^c\gamma^c n^4)$ induced copies of $C_4$, or admits a vertex partition $V(G) = X_1 \cup \dots \cup X_k \cup Y$ with the following properties.
\begin{enumerate}
\item $e(Y) < \alpha n^2$.
\item
$|X_i| \geq 0.1\alpha^3 n$
%$|X_i| = \left\lceil 0.1\alpha^3 n \right\rceil$
and $d(X_i) \geq 1 - \gamma$ for every
$1 \leq i \leq k$.
\item For every $1 \leq i < j \leq k$, the pair $(X_i,X_j)$ is $\gamma$-close to being induced $M_2$-free.
\end{enumerate}
\end{lemma}
\begin{proof}
We prove the lemma with $c = \max(84, 20d)$,
where $d$ is the constant from Lemma \ref{lemma:M2_removal}. We inductively define two sequences of sets, $(V_i)_{i \geq 0}$ and
$(X_i)_{i \geq 1}$. Set $V_0 = V(G)$. At the $i$'th step
(starting from $i = 0$), if $e(V_i) < \alpha n^2$ then we stop.  Note that if we did not stop then $|V_i| \geq \alpha n$.
If $e(V_i) \geq \alpha n^2$ then by Lemma \ref{lemma:clique}, applied to $G[V_i]$ with parameters $\alpha$ and
$\beta = 0.25\gamma^d$, either $G[V_i]$ contains
$\Omega(\alpha^{80} \gamma^{20d} |V_i|^4) \geq \Omega(\alpha^{84}\gamma^{20d}n^4)$ induced copies of $C_4$ or there is
$X_{i+1} \subseteq V_i$ with $|X_{i+1}| \geq 0.1\alpha^2|V_i| \geq 0.1\alpha^3 n$ and $d(X_i) \geq 1 - 0.25\gamma^d$. If the former case happens then the assertion of the lemma holds, so we may assume that the latter case happens, in which case we set
$V_{i+1} = V_i \setminus X_{i+1}$ and continue.
Suppose that this process stops at the $k$'th step for some
$k \geq 0$. Set $Y = V_k$. We clearly have
$V(G) = X_1 \cup \dots \cup X_k \cup Y$. For every $1 \leq i \leq k$ we have
$|X_i| \geq 0.1\alpha^3 n$ and
$d(X_i) \geq 1 - 0.25\gamma^d \geq 1 - \gamma$. Since the process stopped at the $k$'th step, we must have $e(Y) = e(V_k) < \alpha n^2$.

To finish the proof, we show that if Item 3 in the lemma does not hold then $G$ contains at least
$0.5 \cdot 10^{-4}\alpha^{12}\gamma^d n^4$ induced copies of $C_4$.
Assume that for some $1 \leq i < j \leq k$, the pair $(X_i,X_j)$ is $\gamma$-far from being induced $M_2$-free. By Lemma \ref{lemma:M2_removal}, $(X_i,X_j)$ contains at least
$\gamma^{d}|X_i|^2|X_j|^2$ induced copies of $M_2$. Let $(x_i,x_i',x_j,x_j')$ be such a copy, where $x_i,x_i' \in X_i$ and $x_j,x_j' \in X_j$. If $(x_i,x_i'),(x_j,x_j') \in E(G)$ then $x_i,x_i',x_j,x_j'$ span an induced copy of $C_4$. Since $d(X_i),d(X_j) \geq 1 - 0.25\gamma^d$, There are at most $0.5\gamma^d |X_i|^2|X_j|^2$ quadruples of distinct vertices
$(x_i,x_i',x_j,x_j') \in X_i \times X_i \times X_j \times X_j$ for which either
$(x_i,x_i') \notin E(G)$ or $(x_j,x_j') \notin E(G)$. Thus, $G$ contains at least
$0.5\gamma^{d}|X_i|^2|X_j|^2 \geq 0.5 \cdot 10^{-4}\alpha^{12}\gamma^d n^4$ induced copies of \nolinebreak $C_4$.
\end{proof}

We finish this section with the following corollary of the above structure theorem, which will be more convenient to use when proving Theorems \ref{thm:c4} and \ref{thm:chordal} in the next section.

\begin{lemma}\label{lemma:cond_regularity}
There is an absolute constant $c > 0$
%(???Is this the same $c$ from previous lemma??? Yes)
such that for every
$\alpha,\gamma \in (0,1)$, every $n$-vertex graph $G$ either contains $\Omega(\alpha^c\gamma^c n^4)$ induced copies of $C_4$ or there is a graph $G'$ on $V(G)$, a partition
$V(G) = X_1 \cup \dots \cup X_k \cup Y$, where $k \leq 10\alpha^{-3}$, a subset
$Z \subseteq X := X_1 \cup \dots \cup X_k$,
a partition $X \setminus Z = Q_1 \cup \dots \cup Q_q$ which refines
$\{ X_1\setminus Z,\dots,X_k \setminus Z \}$, and subsets $W_i \subseteq Q_i$ with the following properties.
\begin{enumerate}
%\item $k \leq 10\alpha^{-3}$.
\item $G'[X_i \setminus Z]$ is a clique for every $1 \leq i \leq k$, and $G'[Y]$ is an independent set.
\item $|Z| < \alpha n$ and every $z \in Z$ is an isolated vertex in $G'$.
\item In $G'$, the sum of $|Q_i||Q_j|$ over all non-homogeneous pairs $(Q_i,Q_j)$,
$1 \leq i < j \leq q$, is at most $\alpha n^2$.
\item $(W_i,W_j)$ is homogeneous in $G'$ for every $1 \leq i < j \leq q$ and
$|W_i| \geq (\alpha/20)^{4000\alpha^{-6}} |X|$ for every $1 \leq i \leq q$.
\item
$\left| E(G') \triangle E(G) \right| < (2\alpha + \gamma)n^2$ and
$\left| E(G'[X \setminus Z]) \, \triangle \, E(G[X \setminus Z]) \right| < \gamma n^2$.
\end{enumerate}
\end{lemma}
\begin{proof}
The constant $c$ in this lemma is the same as in Lemma \ref{lemma:main}.
Apply Lemma \ref{lemma:main} to $G$ with the given $\alpha$ and $\gamma$. If $G$ contains
$\Omega \left( \alpha^{c}\gamma^{c} n^4\right)$ induced copies of $C_4$ then the assertion of the lemma holds, and otherwise let $X_1,\dots,X_k,Y$ be as in the statement of Lemma \ref{lemma:main}. Note that $k \leq 10\alpha^{-3}$ since $|X_i| \geq 0.1\alpha^3$ for every $1 \leq i \leq k$.
Let $G''$ be the graph obtained from $G$ by making $Y$ an independent set, making $X_1,\dots,X_k$ cliques and making $(X_i,X_j)$ induced $M_2$-free for every $1 \leq i < j \leq k$. By Lemma \ref{lemma:main} we have $|E(G''[Y]) \triangle E(G[Y])| < \alpha n^2$ and
$|E(G''[X]) \triangle E(G[X])| < \gamma \sum_{i=1}^{k}\binom{|X_i|}{2} + \gamma \sum_{i<j}|X_i||X_j| < \gamma n^2$.
%the number of changes inside $Y$ is less than $\alpha n^2$ and the number of changes inside
%$X = X_1 \cup \dots X_k$ is less than
%$\gamma \sum_{i=1}^{k}\binom{|X_i|}{2} + \gamma \sum_{i<j}|X_i||X_j| <
%\gamma n^2$.
We now apply Lemma \ref{lemma:homog_part_strong} to
$G''[X]$ with parameter $\delta = \alpha$ (and with respect to the partition $\{X_1,\ldots,X_k\}$) and obtain
a subset $Z \subseteq X$ of size $|Z| < \alpha |X| \leq \alpha n$, a partition
$X \setminus Z = Q_q \cup \dots \cup Q_q$ which refines
$\{ X_1\setminus Z,\dots,X_k \setminus Z \}$, and subsets $W_i \subseteq Q_i$ such that
$|W_i| \geq (\alpha/2k)^{10k^2}|X| \geq (\alpha^4/20)^{1000\alpha^{-6}}|X| \geq
(\alpha/20)^{4000\alpha^{-6}}|X|$ for every $1 \leq i \leq q$.

Let $G'$ be the graph obtained from $G''$ by making every $z \in Z$ an isolated vertex. Then Item 2 is satisfied. The second part of Item 5 holds because $G'[X \setminus Z] = G''[X \setminus Z]$ and $|E(G''[X]) \triangle E(G[X])| < \gamma n^2$. For the first part of Item 5,
note that
$\left| E(G') \triangle E(G'') \right| < |Z|n < \nolinebreak \alpha n^2$, which implies that
$\left| E(G') \triangle E(G) \right| \leq \left| E(G') \triangle E(G'') \right| + \left| E(G'') \triangle E(G) \right| < (2\alpha + \gamma)n^2$.
Since $G'[X \setminus Z] = G''[X \setminus Z]$ and $G'[Y] = G''[Y]$, it is enough to establish that
Items 1, 3 and 4 hold if $G'$ is replaced by $G''$. For Item 1, this is immediate from the definition of $G''$; for items 3-4, this follows from our choice of $\Q = \{Q_1,\dots,Q_q\}$ and $W_1,\dots,W_q$ via Lemma \ref{lemma:homog_part_strong} (with parameter $\delta = \alpha$).

\end{proof}

\section{Proofs of main results}\label{sec:main}

In this section we prove Theorems \ref{thm:c4} and \ref{thm:chordal}.
The last ingredient we need is the following key lemma.

\begin{lemma}\label{lemma:ind_set}
Let $\F$ be a (finite or infinite) family of graphs such that
\begin{enumerate}
\item $C_4 \in \F$.
\item For every $F \in \F$ and $v \in V(F)$, the neighbourhood of $v$ in $F$ is not a clique.
\end{enumerate}
Suppose $G$ is a graph with vertex partition
$V(G) = X \cup Y$ such that $Y$ is an independent set and $G[X]$ is induced $\F$-free.
Then, if one must add/delete at least $\varepsilon |X||Y|$ of the edges between $X$ and $Y$
to make $G$ induced $\F$-free, then $G$ contains at least
$\frac{\varepsilon^4}{2^8}|X|^2|Y|^2$ induced copies of $C_4$.
\end{lemma}
\begin{proof}
%We assume, for contradiction, that one must change at least
%$\varepsilon |A||B|$ of the edges between $A$ and $B$ in order to turn $G$ into a graph that satisfies $\Prop$.

%For a graph $G'$ on $V(G)$, let $\T(G')$ be the set of triples $(y,u,v)$
%such that $y \in Y$, $u,v \in X$, $u \neq v$,
%$(u,y), (v,y) \in E(G')$ and $(u,v) \notin \nolinebreak E(G')$.

Let us pick for every $y \in Y$ a maximal anti-matching ${\cal M}(y)$ in $G[N_X(y)]$,
that is, a maximal collection of pairwise-disjoint non-edges contained in $N_X(y)$.
%Set $m(y)$ denote the vertices spanning ${\cal M}(y)$ and note that $|M(y)|=|{\cal M}(y)|$.
For every pair of non edges $(u,v),(u',v') \in {\cal M}(y)$, there must be at least one non-edge between $\{u,v\}$ and $\{u',v'\}$, as otherwise
$u,v,u',v'$ would span an induced $C_4$ in $X$, in contradiction to the assumptions that $G[X]$ is induced $\F$-free and $C_4 \in \F$.
Therefore, for every $y$ there are at least $\binom{|{\cal M}(y)|}{2} + |{\cal M}(y)| \geq |{\cal M}(y)|^2/2$ non-edges inside the set $N_X(y)$.
For every $y \in Y$ let $d_2(y)$ denote the number of pairs of distinct vertices in $N_X(y)$ that are non-adjacent. Then the above discussion implies that
every $y \in Y$ satisfies
\begin{equation}\label{eq:good_triples}
d_2(y) \geq \frac{|{\cal M}(y)|^2}{2}\;.
\end{equation}

Let $G'$ be the graph obtained from $G$ by deleting, for every $y \in Y$, all edges going between $y$ and the vertices of ${\cal M}(y)$.
Since ${\cal M}(y)$ is spanned by $2|{\cal M}(y)|$ vertices, we have
\begin{equation}\label{eq:diff}
|E(G') \triangle E(G)| = 2\sum_{y \in Y}|{\cal M}(y)|\;.
\end{equation}
We now claim that $G'$ is induced $\F$-free. Indeed, suppose $U \subseteq V(G)$ spans an induced copy of some $F \in \F$.
Since by assumption $G[X]$ is induced $\F$-free and since $G'[X] = G[X]$, there must be some $y \in U \cap Y$.
Since the neighbourhood of $y$ in $F$ is not a clique and since $G'[Y] = G[Y]$ is an empty graph, there must be $u,v \in U \cap X$ for which
$u,v \in N_X(y)$ and $(u,v) \notin E(G')$. Now, the fact that $u,v$ are connected to $y$ in $G'$ means that neither of them participated in
one of the non-edges of ${\cal M}(y)$. But then the fact that $(u,v) \notin E(G')$ implies that also $(u,v) \notin E(G)$ (because we did not change $G[X]$)
which in turn implies that $(u,v)$ could have been added to ${\cal M}(y)$ contradicting its maximality.

By the assumption of the lemma we thus have
$|E(G') \triangle E(G)| \geq \varepsilon |X||Y|$.
Combining this with \eqref{eq:good_triples}, \eqref{eq:diff} and Jensen's inequality thus gives
\begin{equation*}
\sum_{y \in Y}d_2(y) \geq \frac{1}{2}\sum_{y \in Y}{|{\cal M}(y)|^2} \geq
\frac{1}{2}|Y| \cdot \left( \frac{\sum_{y \in Y}{|{\cal M}(y)|}}{|Y|} \right)^2 =
\frac{1}{2}|Y| \cdot \left( \frac{|E(G') \triangle E(G)|}{2|Y|} \right)^2 \geq
\frac{\varepsilon^2}{8}|X|^2|Y|.
\end{equation*}

For a pair of distinct vertices $u,v \in X$ set $t(u,v)=0$ if $(u,v) \in E(G)$ and otherwise set $t(u,v)$
to be the number of vertices $y \in Y$ connected to both $u$ and $v$. Recalling that $Y$ is an independent set in $G$, we see that $u,v$
belong to at least ${t(u,v) \choose 2}$ induced copies of $C_4$. Hence, $G$ contains at least
\begin{eqnarray*}
\sum_{u,v \in X}{\binom{t(u,v)}{2}} &\geq& \binom{|X|}{2} \cdot \binom{\sum_{u,v \in X}t(u,v)/\binom{|X|}{2}}{2}\\
&=& \binom{|X|}{2} \cdot \binom{\sum_{y \in Y}d_2(y)/\binom{|X|}{2}}{2}\\
&\geq& \frac{|X|^2}{4} \cdot \frac{(\varepsilon^2|Y|/4)^2}{4} = \frac{\varepsilon^4}{2^8}|X|^2|Y|^2,
\end{eqnarray*}
induced copies of $C_4$, where the first inequality is Jensen's, the following equality is double-counting, and the last inequality
uses our above lower bound for $\sum_{y \in Y}d_2(y)$.
\end{proof}

%$\{u,v\} \in \binom{X}{2}$, set
%$T(u,v) = \left\{ y \in Y : (y,u,v) \in \T(G) \right\}$ and
%$t(u,v) = |T(u,v)|$. For every pair of distinct $y,y' \in T(u,v)$, the vertices $y,y',u,v$ span an induced copy of $C_4$ in $G$, since $Y$ is an independent %set. Thus, there are at least
%$\sum_{u,v \in X}{\binom{t(u,v)}{2}}$ induced copies of $C_4$ in $G$. Note that $\sum_{u,v \in X}{t(u,v)} = |\T(G)| \geq \frac{\varepsilon^2}{8}|X|^2|Y|$. By %Jensen's inequality, the number of induced copies of $C_4$ in $G$ is at least
%Note that the graph families $\F = \{C_4\}$ and $\F = \{C_{\ell} : \ell \geq 4\}$ (the family of forbidden induced subgraphs for chordality) satisfy %Conditions 1-2 of Lemma \ref{lemma:ind_set}.

We are now ready to prove Theorems \ref{thm:c4} and \ref{thm:chordal}.

\begin{proof}[Proof of Theorem \ref{thm:c4}]
Set
\begin{equation*}
\alpha = \frac{\varepsilon^6}{2^{11}}, \; \; \;
\gamma =
\frac{1}{2} (\alpha/20)^{16000\alpha^{-6}} (\varepsilon/2)^4.
\end{equation*}
and notice that $\gamma \geq 2^{-(1/\varepsilon)^{c'}}$ for some absolute constant $c'$.
% We choose $\varepsilon_0$ to be small enough to guarantee that $\alpha < \alpha_0$, where $\alpha_0$ is the constant from Lemma \ref{lemma:main}.
We apply Lemma \ref{lemma:cond_regularity} to $G$ with the $\alpha$ and $\gamma$ defined above. If $G$ contains
$\Omega\left(\alpha^{c}\gamma^{c} n^4\right)$ induced copies of $C_4$ then we are done. Otherwise, let $G'$, $X = X_1 \cup \dots \cup X_k$, $Y$, $Z$,
$\Q = \{Q_1,\dots,Q_q\}$ and $W_i \subseteq Q_i$ be as in Lemma \ref{lemma:cond_regularity}.
%Since $2\alpha + \gamma < \dots$, Item 3 in Corollary \ref{cor:cond_regularity} implies that $G'$ is $\dots$-far from being induced $C_4$-free.
Let $G''$ be the graph obtained from $G'$ by doing the following: for every $1 \leq i < j \leq q$, if $(W_i,W_j)$ is a complete (resp. empty) bipartite graph then we turn $(Q_i,Q_j)$ into a complete (resp. empty) bipartite graph. By Item 4 in Lemma \ref{lemma:cond_regularity}, one of these options holds. By Item 3 in Lemma \ref{lemma:cond_regularity}, the number of changes made is at most $\alpha n^2$. By Item 5 in Lemma \ref{lemma:cond_regularity} we have
$\left| E(G'') \triangle E(G) \right| \leq
\left| E(G'') \triangle E(G') \right| + \left| E(G') \triangle E(G) \right| <
(3\alpha + \gamma)n^2 < \frac{\varepsilon}{2} n^2$, implying that $G''$ is $\frac{\varepsilon}{2}$-far from being induced $C_4$-free. Note that
$|X \setminus Z| \geq \frac{\varepsilon}{2}n$, as otherwise deleting all edges incident to the vertices of $X \setminus Z$ would make $G''$ an empty graph (and hence induced $C_4$-free) by deleting $|X \setminus Z| \cdot n \leq \frac{\varepsilon}{2}n^2$ edges.

Let us assume first that $G''[X \setminus Z]$ contains an induced copy of $C_4$, say on the vertices $v_1,v_2,v_3,v_4$. For $1 \leq s \leq 4$, let $i_{s}$ be such that $v_{s} \in Q_{i_{s}}$.
%Note that $i_s \neq i_t$ for $s \neq t$ because $Q_1,\dots,Q_q$ are cliques, all pairs $(Q_i,Q_j)$ are homogeneous in $G''$, and $C_4$ does not contain two adjacent vertices with the same neighbourhood.
It is easy to see that by the definition of $G''$, every quadruple
$(w_1,\dots,w_4) \in W_{i_{1}} \times W_{i_{2}} \times W_{i_{3}} \times W_{i_4}$ spans an induced copy of $C_4$ {\em in the graph $G'$}. By Item 4 in Lemma \ref{lemma:cond_regularity}, $G'$ contains
$$|W_{i_{1}}| \cdot |W_{i_{2}}| \cdot |W_{i_{3}}| \cdot |W_{i_{4}}| \geq
(\alpha/20)^{16000\alpha^{-6}}|X|^4 \geq
(\alpha/20)^{16000\alpha^{-6}} (\varepsilon/2)^4n^4 = \nolinebreak 2\gamma n^4$$ induced copies of $C_4$. By Item 5 in Lemma \ref{lemma:cond_regularity}, $G[X \setminus Z]$ and $G'[X \setminus Z]$ differ on less than $\gamma n^2$ edges, each of which can participate in at most $n^2$ induced copies of $C_4$. Thus, $G$ contains at least $\gamma n^4$ induced copies of $C_4$, as required.

%Observe that $|X| \geq \frac{\varepsilon }{2}n$, as otherwise $G$ contains at most
%$e(Y) + |X|n < \varepsilon n^2$ edges, and is thus $\varepsilon$-close to being induced $C_4$-free.

From now on we assume that $G''[X \setminus Z]$ is induced $C_4$-free, implying that $G''[X]$ is induced $C_4$-free (as every $z \in Z$ is isolated in $G''$).
Since $G''$ is $\frac{\varepsilon}{2}$-far from being induced $C_4$-free, one cannot make $G''$ induced $C_4$-free by adding/deleting less than $\frac{\varepsilon}{2} n^2 \geq \varepsilon |X||Y|$ edges between $X$ and $Y$. In particular, we have $|X||Y| \geq \varepsilon n^2$.
%because otherwise one could delete all edges between $X$ and $Y$ and thus make $G''$ induced $C_4$-free with less than $\frac{\varepsilon}{2} n^2$ edge changes.
Notice that the conditions of Lemma \ref{lemma:ind_set} hold (with respect to the family ${\cal F}=\{C_4\}$) since $G''[Y] = G'[Y]$ is an independent set (by Item 1 in Lemma \ref{lemma:cond_regularity}) and $G''[X]$ is induced $C_4$-free by assumption.
By Lemma \ref{lemma:ind_set}, $G''$ contains at least
$\frac{\varepsilon^4}{2^8} |X|^2|Y|^2 \geq \frac{\varepsilon^6}{2^8}n^4 = 8\alpha n^4$ induced copies of $C_4$.
Since $\left| E(G'') \triangle E(G) \right| < (3\alpha + \gamma)n^2 < 4\alpha n^2$, at least
$4\alpha n^4 = \frac{\varepsilon^6}{2^{9}}n^4$ of these copies are also present in $G$. This completes the proof of the theorem.
\end{proof}

\begin{proof}[Proof of Theorem \ref{thm:chordal}]
Set
\begin{equation*}
\alpha = \frac{\varepsilon^{6}}{2^{11}}, \; \; \;
\gamma =
\frac{1}{2} (\alpha/20)^{10^5\alpha^{-9}} (\varepsilon/2)^{20\alpha^{-3}}.
\end{equation*}
and notice that $\gamma \geq 2^{-(1/\varepsilon)^{c'}}$ for some absolute constant $c'$.
% We choose $\varepsilon_0$ to be small enough to guarantee that $\alpha < \alpha_0$, where $\alpha_0$ is the constant from Lemma \ref{lemma:main}.
As in the proof of Theorem \ref{thm:c4}, we apply Lemma \ref{lemma:cond_regularity} to $G$ with the $\alpha$ and $\gamma$ defined above. If $G$ contains
$\Omega \left( \alpha^{c}\delta^{c} n^4\right)$ induced copies of $C_4$ then we are done. Otherwise, let $G'$, $X = X_1 \cup \dots \cup X_k$, $Y$, $Z$,
$\Q = \{Q_1,\dots,Q_q\}$ and $W_i \subseteq Q_i$ be as in Lemma \ref{lemma:cond_regularity}.
%Since $2\alpha + \gamma < \dots$, Item 3 in Corollary \ref{cor:cond_regularity} implies that $G'$ is $\dots$-far from being induced $C_4$-free.

Let $G''$ be the graph obtained from $G'$ by doing the following: for every $1 \leq i < j \leq q$, if $(W_i,W_j)$ is a complete (resp. empty) bipartite graph then we make $(Q_i,Q_j)$ a complete (resp. empty) bipartite graph.
As in the proof of Theorem \ref{thm:c4}, $G''$ is $\frac{\varepsilon}{2}$-far from being chordal, and we have $|X \setminus Z| \geq \frac{\varepsilon}{2}n$.
%By Item 3 in Lemma \ref{lemma:cond_regularity}, the number of changes made is at most $\alpha n^2$. By Item 5 in Lemma \ref{lemma:cond_regularity} we have
%$\left| E(G'') \triangle E(G) \right| \leq
%\left| E(G'') \triangle E(G') \right| + \left| E(G') \triangle E(G) \right| <
%(3\alpha + \gamma)n^2 < \frac{\varepsilon}{2} n^2$, implying that $G''$ is $\frac{\varepsilon}{2}$-far from being chordal.
%As in the proof of Theorem \ref{thm:C4}, we have $|X| \geq \frac{\varepsilon}{2}n$.

Assume first that $G''[X \setminus Z]$ is not chordal, namely that it contains an induced cycle
$C = v_1 \dots v_{\ell}$ of length $\ell \geq 4$. By Item 1 in Lemma \ref{lemma:cond_regularity}, $G''[X_i \setminus Z] = G'[X_i \setminus Z]$ is a clique for every $1 \leq i \leq k$.
Since the cycle $C$ does not contain a triangle, it can contain at most 2 vertices from each of these cliques, implying that
$\ell = |C| \leq 2k \leq 20\alpha^{-3} = O(\varepsilon^{-18})$. The bound on $k$ comes from Lemma \ref{lemma:cond_regularity}.
For $1 \leq s \leq \ell$, let $i_s$ be such that
$v_s \in Q_{i_s}$. It is easy to see that by the definition of $G''$, $\ell$-tuple
$(w_1,\dots,w_{\ell}) \in W_{i_{1}} \times \dots \times W_{i_{\ell}}$ spans an induced $\ell$-cycle {\em in the graph $G'$}. By Item 4 in Lemma \ref{lemma:cond_regularity}, $G'$ contains
%$$|W_{i_{1}}| \dots |W_{i_{\ell}}| \geq (40\alpha)^{5600\ell/\alpha^6}|X|^{\ell} \geq
%(40\alpha)^{10^6/\alpha^{9}} (\varepsilon/2)^{20/\alpha^3}n^{\ell} = 2\gamma n^4$$
$$\prod\limits_{j=1}^{\ell}|W_{i_j}| \geq (\alpha/20)^{4000\alpha^{-6}\ell}|X|^{\ell} \geq
(\alpha/20)^{10^5\alpha^{-9}} (\varepsilon/2)^{20\alpha^{-3}}n^{\ell} = 2\gamma n^{\ell}$$
induced copies of $C_{\ell}$. By Item 5 in Lemma \ref{lemma:cond_regularity}, $G[X]$ and $G'[X]$ differ on less than $\gamma n^2$ edges, each of which can participate in at most $n^{\ell - 2}$ induced copies of $C_{\ell}$. Thus, $G$ contains at least $\gamma n^{\ell}$ induced copies of $C_{\ell}$, as required.

We now assume that $G''[X]$ is chordal. Since $G''$ is $\frac{\varepsilon}{2}$-far from being chordal, one must add/delete at least $\frac{\varepsilon}{2}n^2 \geq \varepsilon |X||Y|$ of the edges between $X$ and $Y$ to make $G''$ chordal. In particular, we have $|X||Y| \geq \varepsilon n^2$.
Note that the family $\mathcal{F} = \{C_{\ell} : \ell \geq 4\}$, i.e. the family of forbidden induced subgraphs for chordality, satisfies Conditions 1-2 of Lemma \ref{lemma:ind_set}.
Observe that Lemma \ref{lemma:ind_set} is applicable to $G''$
(with respect to the family $\F = \{C_{\ell} : \ell \geq 4\}$), as $G''[Y] = G'[Y]$ is an independent set (by Item 1 in Lemma \ref{lemma:cond_regularity}), and $G''[X]$ is induced $\mathcal{F}$-free (i.e. chordal) by assumption.
By Lemma \ref{lemma:ind_set},
$G''$ contains at least
$\frac{\varepsilon^4}{2^8} |X|^2|Y|^2 \geq \frac{\varepsilon^6}{2^{8}}n^4 = 8\alpha n^4$ induced copies of $C_4$. Since $\left| E(G'') \triangle E(G) \right| < 4\alpha n^2$, at least
$4\alpha n^4 = \frac{\varepsilon^6}{2^{9}}n^4$ of these copies are also present in $G$.
\end{proof}

\section{An impossibility result}\label{sec:example}

In this section we prove Theorem \ref{thm:hard0}.
It will in fact be more convenient to prove the following equivalent statement.

\begin{theorem}\label{thm:hard}
For every function $g : (0,\frac{1}{2}) \rightarrow \mathbb{N}$ there is a graph family $\F$ which contains $C_4$ and there is a sequence $\{\varepsilon_k\}_{k=1}^{\infty}$ with $\varepsilon_k > 0$ and $\varepsilon_k \rightarrow 0$, such the following holds. For every $k \geq 1$ and $n \geq n_0(k)$ there is an $n$-vertex graph $G$ which is $\varepsilon_k$-far from being induced $\F$-free, but still every induced subgraph of $G$ on $g(\varepsilon_k)$ vertices is induced $\F$-free.
\end{theorem}
\noindent
We will need the following theorem due to Erd\H{o}s \cite{E}.
\begin{theorem}\label{thm:k_uniform_Zarankiewicz}
For every integer $f$ there is
$n_{\ref{thm:k_uniform_Zarankiewicz}} = n_{\ref{thm:k_uniform_Zarankiewicz}}(k,f)$
such that every $k$-uniform hypegraph with $n \geq n_{\ref{thm:k_uniform_Zarankiewicz}}$ vertices and $n^{k - f^{1-k}}$ edges contains a complete $k$-partite $k$-uniform hypergraph with $f$ vertices in each part.
\end{theorem}

For integers $k,f \geq 1$, let $B_{k,f}$ be the graph obtained by replacing each vertex of the cycle $C_k$ by a clique of size $f$,
and replacing each edge by a complete bipartite graph.

\begin{lemma}\label{lem:clique_blowup}
For every pair of integers $k \geq 3$ and $f \geq 1$ there is
$n_{\ref{lem:clique_blowup}} = n_{\ref{lem:clique_blowup}}(k,f)$ such that for every $n \geq n_{\ref{lem:clique_blowup}}$, the graph $B_{k,n/k}$ is $\frac{1}{2k^2}$-far from being induced $\{C_4,B_{k,f}\}$-free.
\end{lemma}
\begin{proof}
Let $V_1,\dots,V_k$ be the sides of $G := B_{k,n/k}$ (each a clique of size $n/k$). Let $G'$ be a graph obtained from $G$ by adding/deleting at most $\frac{v(G)^2}{2k^2} = \frac{n^2}{2k^2}$ edges. Our goal is to show that $G'$ is not induced $\{C_4,B_{k,f}\}$-free. Let $H$ be the $k$-partite $k$-uniform hypergraph with parts $V_1,\dots,V_k$ whose edges are all $k$-tuples
$(v_1,\dots,v_k) \in V_1 \times \dots \times V_k$ such that
$v_1v_2\dots v_kv_1$ is an induced cycle in $G'$. Note that in $G$, every such $k$-tuple spans an induced cycle, and that adding/deleting an edge can destroy at most $\left( \frac{n}{k} \right)^{k-2}$ such cycles. Thus, $G'$ contains at least
$\left( \frac{n}{k} \right)^k -
\frac{n^2}{2k^2}\left( \frac{n}{k} \right)^{k-2} =
\frac{1}{2}\left( \frac{n}{k} \right)^k$ of these induced cycles, implying that $e(H) \geq \frac{1}{2}\left( \frac{n}{k} \right)^k$. For a large enough $n$ we have
$\frac{1}{2}\left( \frac{n}{k} \right)^k \geq n^{k - f^{1-k}}$ and
$n \geq n_{\ref{thm:k_uniform_Zarankiewicz}}(k,f)$. Thus, by Theorem \ref{thm:k_uniform_Zarankiewicz}, $H$ contains a complete $k$-partite $k$-uniform hypergraph with parts $U_i \subseteq V_i$, each of size $f$. This means that in the graph $G'$, $(U_i,U_j)$ is a complete bipartite graph if $j - i \equiv \pm 1 \pmod{k}$ and an empty bipartite graph otherwise. If $G'[U_i]$ is a clique for every $1 \leq i \leq k$ then
$U_1 \cup \dots \cup U_k$ spans an induced copy of $B_{k,f}$ in $G'$. Suppose then that $U_i$ is not a clique for some $1 \leq i \leq k$, say $i = 1$, and let $x,y \in U_1$ be such that $(x,y) \notin E(G')$. Then for every
$z \in U_2$ and $w \in U_k$, $\{x,y,z,w\}$ spans an induced copy of $C_4$ in $G'$. Thus, in any case $G'$ is not induced $\{C_4,B_{k,f}\}$-free.
\end{proof}
\begin{proof}[Proof of Theorem \ref{thm:hard}]
For $k \geq 5$ put $\varepsilon_k = \frac{1}{2k^2}$ and
$f_k = g(\varepsilon_k)$. We will show that the family
$\F = \{C_4\} \cup \{ B_{k,f_k} : k \geq 5\}$ satisfies the requirement. Let $k \geq 5$, let $n \geq n_{\ref{lem:clique_blowup}}(k,f_k)$ and set
$G = B_{k,n/k}$. By Lemma \ref{lem:clique_blowup}, $G$ is $\varepsilon_k$-far from being induced $\{C_4,B_{k,f_k}\}$-free. Since $C_4, B_{k,f_k} \in \F$, we get that $G$ is $\varepsilon_k$-far from being induced $\F$-free.

We claim that for every $4 \leq \ell < k$, $G$ is induced $C_{\ell}$-free. Suppose, for the sake of contradiction, that $x_1,\dots,x_{\ell},x_1$ is an induced $\ell$-cycle in $G$. Let $V_1,\dots,V_k$ be the sides of $G = B_{k,n/k}$. If $\left| \{x_1,\dots,x_\ell\} \cap V_i \right| \leq 1$ for every $1 \leq i \leq k$ then $x_1,\dots,x_{\ell}$ are contained in an induced path, which is impossible. So there is some $1 \leq i \leq k$ for which $\left| \{x_1,\dots,x_\ell\} \cap V_i \right| \geq 2$. Suppose without loss of generality that $x_1,x_2 \in V_1$ (recall that $V_1,\dots,V_k$ are cliques). Then $x_3 \in V_2$ or $x_3 \in V_k$, and in either case $x_1,x_2,x_3$ span a triangle, a contradiction.
%On the other hand, it is easy to see that $C_{\ell}$ is not an induced subgraph of $G$ for any
%$4 \leq \ell < k$ (??this was $\neq$. do we need to explain this??).

We conclude that the smallest
$F \in \mathcal{F}$ which is an induced subgraph of $G$, is $F = B_{k,f_k}$.
Thus, every induced subgraph of $G$ on less than $v(B_{k,f_k}) = k \cdot g(\varepsilon_k)$
vertices is induced $\mathcal{F}$-free, completing the proof.
\end{proof}


\begin{thebibliography}{99}
	
\bibitem{subgraphs}
N. Alon, Testing subgraphs in large graphs, Random Structures and Algorithms 21 (2002), 359-370.
	
	
\bibitem{AFKS}
N. Alon, E. Fischer, M. Krivelevich and M. Szegedy, Efficient testing of large graphs, Combinatorica 20 (2000), 451-476.
	
	
\bibitem{AFN}
N. Alon, E. Fischer, and I. Newman, Testing of bipartite graph properties, SIAM
Journal on Computing 37 (2007), 959-976.

%\bibitem{AFNS}
%N. Alon, E. Fischer, I. Newman and A. Shapira,
%A combinatorial characterization of the testable graph properties: it's all about regularity,
%SIAM Journal on Computing 39 (2009), 143-167.
	
	
\bibitem{AF}
N. Alon and J. Fox, Easily testable graph properties, Combin. Probab. Comput., 24 (2015), 646-657.
	
	
\bibitem{induced}
N. Alon and A. Shapira, A characterization of easily testable induced subgraphs, Combin. Probab.Comput. 15 (2006), 791-805.
	

\bibitem{hereditary}
N. Alon and A. Shapira, A characterization of the (natural) graph properties testable with one-sided error,
SIAM Journal on Computing 37 (2008), 1703-1727.

	
\bibitem{Prob_Method}
N. Alon and J. H. Spencer, {\bf The Probabilistic Method}, 3rd ed., Wiley, 2008.
	

\bibitem{ConlonFo12}
D.~Conlon and J.~Fox,
{Bounds for graph regularity and removal lemmas},
GAFA \textbf{22} (2012), 1191--1256.

\bibitem{ConlonFo13}
D.~Conlon and J.~Fox,
{Graph removal lemmas},
Surveys in Combinatorics, Cambridge university press, 2013, 1-50.

\bibitem{E}
P. Erd\H{o}s, On extremal problems of graphs and generalized graphs, Israel J. Math. 2, 1964, 183-190.

\bibitem{Erdos84}
P. Erd\H{o}s, On some problems in graph theory, combinatorial analysis and combinatorial number theory. In {\em Graph
theory and combinatorics (Cambridge, 1983)}, pages 1-17. Academic Press, London, 1984.

\bibitem{Fox11}
J.~Fox,
{A new proof of the graph removal lemma},
{Ann. of Math.} \textbf{174} (2011), 561--579.

\bibitem{GS}
L. Gishboliner and A. Shapira, Removal lemmas with polynomial bounds, Proc. of STOC 2017.

\bibitem{Goldreich}	
O. Goldreich, {\bf Introduction to Property Testing}, Forthcoming book, 2017.

	
\bibitem{GGR}
O. Goldreich, S. Goldwasser, and D. Ron, Property testing and its connection to learning
and approximation, J. ACM 45 (1998), 653-750.

\bibitem{Gowers}
T.~Gowers,
{Lower bounds of tower type for Szemer\'edi's uniformity lemma},
{Geom. Funct. Anal.} \textbf{7} (1997), 322--337.		
	
\bibitem{GHS}
A. Gy\'{a}rf\'{a}s, A. Hubenko and J. Solymosi, Large cliques in $C_4$-free graphs, Combinatorica, 22 (2002), 269-274.

\bibitem{Lovasz}
L. Lov\'{a}sz, \textbf{Large networks and graph limits} (Vol. 60), Providence: American Mathematical Society (2012).

\bibitem{MoshkovitzSh16}
G.~Moshkovitz and A.~Shapira,
{A sparse regular aproximation lemma}, Transactions of the AMS, to appear.

\bibitem{Steger}
H. Pr\"{o}mel and A. Steger,
Excluding induced subgraphs: quadrilaterals,
Random Structures and Algorithms 2 (1991), 55-71.


\bibitem{RD}
V. R\"{o}dl and R. Duke, On graphs with small subgraphs of large chromatic number, Graphs and Combinatorics 1 (1985), 91-96.

\bibitem{RodlSh}
V. R\"{o}dl and M. Schacht, Regularity lemmas for graphs,
Fete of Combinatorics and Computer Science, Bolyai Soc. Math. Stud., 20 (2010), 287--325.

%\bibitem{RS}
%R. Rubinfeld and M. Sudan, Robust charachterization of polynomials with applications to program testing, SIAM Journal on Computing 25 (1996), 252-271.

\bibitem{RuzsaSz76}
I.Z.~Ruzsa and E.~Szemer\'edi,
{Triple systems with no six points carrying three triangles}, in Combinatorics (Keszthely, 1976), Coll. Math.
Soc. J. Bolyai 18, Volume II, 939-945.

\bibitem{Szemeredi78}
E.~Szemer\'edi,
{Regular partitions of graphs},
In: {\em Proc.\ Colloque Inter.\ CNRS}, %(J.~C.~Bermond, J.~C.~Fournier, M.~Las~Vergnas and D.~Sotteau, eds.),
1978, 399--401.

\end{thebibliography}
\end{document}